\documentclass[10pt]{amsart}

\usepackage{amssymb}
\usepackage{soul}
\usepackage{bm}
\usepackage{graphicx}
\usepackage{psfrag}
\usepackage[usenames,dvipsnames]{xcolor}
\usepackage{enumerate}
\usepackage{multirow}
\usepackage{url}

\usepackage{subcaption}

\usepackage{comment}

\usepackage{algpseudocode}

\usepackage{mathtools}

\usepackage{xy}
\input xy
\xyoption{all}

\usepackage{tikz}
\usepackage{tikzscale}
\usetikzlibrary{arrows}

\newcommand{\excise}[1]{}

\newtheorem{thm}{Theorem}[section]

\theoremstyle{definition}
\newtheorem{alg}[thm]{Algorithm}

\newtheorem{remark}[thm]{Remark}

\numberwithin{equation}{section}

\newcommand\NN{\mathbb{N}}
\newcommand\OO{{\mathcal O}}

\newcommand\ZZ{\mathbb{Z}}

\newcommand\cP{{\mathcal P}}
\newcommand\powerset{{\cP}}

\newcommand{\quotes}[1]{{``#1"}}

\newcommand{\Zz}{\mathsf Z}

\begin{document}

\mbox{}
\title
[GPU-accelerated factorization sets in numerical semigroups]
{GPU-accelerated factorization sets in numerical semigroups via parallel bounded lexicographic streams}
\author[Thomas Barron]{Thomas Barron}
\address{Denver, CO 
}
\email{t@tbarron.xyz}

\date{\today}

\begin{abstract}
We describe a method for parallelizing the lexicographic enumeration algorithm for the factorization set of an element in a numerical semigroup via bounds.
This enables the use of GPU and distributed computing methods.
We provide a CUDA implementation with measured runtimes.
\end{abstract}
\maketitle
\vspace{-1em}
\section{Introduction}
\label{sec:intro}

The factorization set problem in numerical semigroups\footnote{Properly speaking, a \textit{numerical semigroup} is a cofinite subsemigroup of the naturals with addition; and furthermore, each numerical semigroup has a unique minimal set of generators. However, the present algorithm does not require its inputs to be the minimal generating set of a numerical semigroup: it can accept any stream of positive integer inputs, regardless of order, setwise coprimality, or minimality, including outright repeated generators.}
concerns itself with computing the set of all nonnegative solutions $(a_1,...,a_d)$ to the heterogeneous linear Diophantine equation $a_1g_1 + ... + a_dg_d = n$, for fixed nonnegative $g_i$ and $n$; $d$ is called the \textit{dimension}. That is, with respect to a set of generators $(g_1,...,g_d) \in \NN^d$, the \textbf{factorization set} of an element $n \in \NN$ is defined as:
\[\Zz : \NN \times \NN^d \to \powerset(\NN^d)\]
\[\Zz\left(n, (g_1, ..., g_d)\right) := \left\{(a_1,...,a_d) \middle| \sum_i a_ig_i = n\right\}\]

Finding the set of nonnegative solutions to general linear Diophantine equations, which may have variables on both sides of the equation, is a long-studied question in numerical computing. \cite{performance} provides a brief comparison of efficient single-threaded algorithms for computing such. Contemporary approaches can also include parallel techniques \cite{utk} \cite{perekhod}.

Restricting to the numerical semigroup case, with all variables on the left, allows for algorithmic optimizations as the solution set is much smaller, contained in the first orthant. Dynamic programming techniques \cite{dynamic} are possible in this restricted case. 

\textbf{Lexicographic enumeration} of nonnegative integer solutions to a diophantine linear equation is a particularly well-described technique, with algorithms dating to at least 1977 \cite{lex1978}.

In this work we describe, and provide a reference CUDA implementation for, a method of \linebreak\textbf{parallelizing} lexicographic enumeration of $\Zz(n, (g_1,...,g_d))$, by assigning many workers a distinct upper- and lower-bounded work slice of the final output.

In particular we note that the parallelization of lexicographic succession we describe is of a different nature than that described by Perekhod in \cite{perekhod}. 
The parallelism of \cite{perekhod} is limited to a parallel summing step and a parallel comparison step during the evaluation of the successor function; it evaluates the stream of successor values sequentially.
In comparison, our parallelism splits the result set, that is the results of the successor function, among multiple workers.\footnote{In the language of \cite{branchandbound} for parallel branch-and-bound algorithms in integer optimization, \cite{perekhod} exhibits \emph{type 1} parallelism, whereas our implementation exhibits \emph{type 2}.}

\section{Streams}

Informally, a \textbf{stream} is a potentially infinite list-like data type whose elements are produced and/or consumed one at a time. 

A stream is emitted by a \textit{producer} and is transformed or consumed by a \textit{consumer}. One consumer may have the ability to consume multiple streams; in the case of GPU parallelism, the CPU has the ability to consume the results of multiple streams produced in parallel on the GPU.  

The $\mathsf{nextCandidate}$ stream we will describe will be have the property that its next stream output value depends only on the previous stream output value (where relevant internal state members are included in the returned struct). It can therefore be described as a data struct $State$ and a pure function $nextState: State \to State$, although in practice this function may be implemented in-place to avoid copies.

As we define a stream producer for factorizations, potential consumers include: saving the factorizations to disk or in memory, to retain the entire result set; incrementing a counter for each result, to retain only the cardinality of the result set; or setting a boolean variable based on a predicate's value for each result, to retain whether any elements of the factorization set satisfy the predicate.

\section{Lexicographic enumeration}

An algorithm is called \quotes{greedy} if at each step it chooses the locally optimal solution, that is, the best solution available at each step. 


In our case the presented algorithm $\mathsf{nextCandidate}$ is greedy in that it attempts to produce the \linebreak lexicographically next-greatest factorization; if what it produces is not a factorization, it is a \linebreak\quotes{candidate}$\in\NN^d$ that is lexicographically greater than the next factorization.

\subsection{The stream data struct}

In the following structs, $\NN$ is represented in the usual fixed-bit-depth representation using $B$ bits, such that all values involved are less than $2^B$. Arbitrary-precision integer implementations such as Java's BigInteger \cite{javabigint} would enable the same function to handle all possible sizes of inputs without the need to fix a precision, with less predictable memory and performance characteristics.

\subsubsection{The (single-worker) lexicographic state struct}
$$\mathsf{type\ } LexicographicState := \{\ element: \NN,\ generators: \NN^d,$$
$$previousCandidate: \NN^d,\ wasValid: bool,\ endOfStream: bool\ \}$$

\subsubsection{Bound for enabling division of work}
The preceding struct enables a single individual worker to iteratively produce the entire desired set of factorizations. In order to enable multiple workers working in a parallel or distributed manner on the same factorization set, while maintaining a uniqueness guarantee, we can provide the stream struct with a \emph{bound}, which when reached will denote that the given worker's job is complete. 
$$\mathsf{type\ } ParallelLexicographicState := \{\ element: \NN,\ generators: \NN^d,\ bound: \NN^d,$$
$$previousCandidate: \NN^d,\ wasValid: bool,\ endOfStream: bool\ \}$$

The single-worker case, as well as the case of the first worker from which all other workers will split their work, can be represented by a bound of $(0,...,0)$ as this is, trivially, lexicographically less than any valid factorization.

The matter of the heuristic for choosing this bound will be addressed in a later section.

\subsection{The lexicographic next-candidate algorithm}

We compute candidates in lexicographic decreasing order. We note that this technique could readily be modified to proceed in increasing order.

\subsubsection{First candidate}
The first returned candidate is defined to be $(x, 0, ..., 0)$ where \linebreak$x = (n + g_1 - 1) / g_1$, where $/$ denotes integer division with remainder discarded. We note that this is $n / g_1$ if $n \equiv 0 \bmod g_1$, or $(n/g_1) + 1$ if $n / g_1$ has a remainder; and thus that $x g_1 \ge n$, and that $x$ is the \quotes{minimal} $x$ such that $xg_1 \ge n$.

If valid, then this factorization is lexicographically the greatest:
for any other factorization $(a_1,...,a_d)$ must have nonzero $a_i$ in some index other than the first, so $xg_1 = a_1g_1 + y$ where $y = \sum_2^d{a_i}$ is nonzero, implying its first index $a_1$ to be less than $x$, and thus $\vec a$ is lexicographically less than $(x,0,...,0)$.

If invalid, then this is lexicographically greater than any factorization: for no factorization $\vec a$ can be of the form $(x,a_2,...,a_d)$, as $\sum_i a_ig_i \ge xg_1 > n$. Thus any valid factorization will have $a_1 < x$ and will be lexicographically less than $(x,0,...,0)$.

\subsubsection{Subsequent candidates}

We define $\phi:\vec a \mapsto \sum_i a_ig_i$. We here define a function \linebreak$\mathsf{decrementAndSolve}: (\vec a : \NN^d,\ i: \NN) \to \NN^d$, which is defined to decrement $a_i$ and set $a_{i+1}$ to the minimal $m$ such that $n-\phi(\vec a) - mg_{i+1} \le 0$. We note that this value is $(n-\phi(\vec a)) / g_1$ if $(n-\phi(\vec a)) \equiv 0 \bmod g_{i+1}$, or $((n-\phi(\vec a))/g_1) + 1$ if $(n-\phi(\vec a)) / g_{i+1}$ has a remainder.

For subsequent candidates, use Algorithm 3.1, which can be described as follows: Take index $1 \le i < d$ maximal such that $a_i > 0$ (note the exclusion of index $d$);  set $a_d = 0$; 
perform $\mathsf{decrementAndSolve}(\vec a, i)$;
and return $\vec a$.

\subsubsection{End of stream}
If the previous candidate was of the form $(0,...,0,x)$, then it returns \linebreak$endOfStream = true$. Similarly, if the candidate it would return is less-or-equal than the worker's \emph{bound}, then it returns $endOfStream = true$.

\begin{alg}[The lexicographic greatest next-candidate algorithm]\ \\
$\mathsf{nextCandidate}:\ ParallelLexicographicState \to ParallelLexicographicState$\\
Input: $(n,\ (g_1,...,g_d),\ (a_1,...,a_d),\ (b_1,...,b_d),\ wasValid,\ endOfStream)$\\
Output: The input variables, modified (the function operates in-place)
\begin{enumerate}
    \item $\mathsf{if}\ (endOfStream)\ return;$ \footnote{This step allows the function to safely no-op in the case that the current worker didn't successfully receive a fresh slice of work during work division; see \quotes{Division of labor} section. }
   \item $\mathsf{let}\ i=0;\ \mathsf{for}\ (\mathsf{let}\ j = 1;\ j\le dim;\ j++) \{\ \mathsf{if}\ (a_j > 0)\ i = j\ \}$. \\(Finding the rightmost nonzero index excluding the final.)
   \item $\mathsf{if}\ (i=0) \{\ endOfStream=true;\ return;\ \}$, detecting if the last candidate was the final
    \item $a_d = 0$, resetting the final index to zero in case it was nonzero from having been solved for previously;
   \item $a_i --$ , which is the first step of $\mathsf{decrementAndSolve}(\vec a, i)$, consisting of steps 5-11
   \item $\mathsf{let}\ p = n - \phi(\vec a)$, the element to attempt to factor by $g_{i+1}$
   \item $\mathsf{let}\ m = p / g_{i+1}$, where division denotes integer division with remainder discarded.
   \item $\mathsf{let}\ r = p - m*g_{i+1}$, the remainder;
   \item $\mathsf{let}\ wasValid = true$, temporarily;
   \item $\mathsf{if}\ (r \neq 0) \{\ m ++;\ wasValid=false; \} $
   \item $a_{i+1} = m$
   \item $\mathsf{if}\ ((a_1,...,a_d) \leq_{lex} bound)\ \{\ endOfStream = true\ \}$
   \item Return $(n,\ (g_1,...,g_d),\ (a_1,...,a_d),\ (b_1,...,b_d),\ wasValid,\ endOfStream)$
\end{enumerate}
\end{alg}

We can directly observe the space complexity of (one iteration of) $\mathsf{nextCandidate}$ by summing the space used by the inputs, outputs, and any temporary variables allocated in the function. In this case, we have one $\NN$ ($B$ bits), three $\NN^d$ ($3*d*B)$, and 2 bools (2 bits) in the inputs, plus 5 temporary $\NN$ ($i,j,p,m,r)$ and one temporary bool, summing to $(3d + 6)B + 3$ bits, for $\OO(d*B)$ space complexity.

\begin{thm}
Let $\vec z$ be a $previousCandidate$ and let $\vec a = \mathsf{nextCandidate}(\vec z)$. Assume that $z_d$ is \quotes{minimal} in that there exists no valid factorization $(z_1,...,z_{d-1},x)$ for any $x<z_d$. Then the next candidate algorithm produces either the lexicographically next factorization, or else, if not a valid factorization, a candidate that is greater than the next factorization.
\end{thm}
\begin{proof}
Let $i$ be from step 2 above. Then $\vec z = (z_1,...,z_i,[0,...,0],z_d)$ and
$\vec a = (z_1,...,z_i - 1, a_{i+1},[0,...,0])$, with $z_d$ possibly zero, and possibly zero zeroes. We show that there lies no factorization $\vec x$ between $\vec a$ and $\vec z$.

As $\vec a$ and $\vec z$ share the first $i-1$ indices, a factorization $\vec x$ lying in between them must as well. In index $i$, $z_i-1 \le x_i \le z_i$, giving two cases.

If $x_i = z_i$, then $\vec x = (z_1,..,z_i,[0,...,0],x_d)$ for some $x_d < z_d$. But, per the assumption, there exists no such factorization as $z_d$ was \quotes{minimal}.

If $x_i = z_i - 1$, then $\vec x = (z_1,...,z_i-1, x_{i+1},...,x_d)$. We note that this shares the first $i$ coordinates with $a_i$. We note $x_{i+1} \ge a_{i+1}$ as we assumed $\vec a <_{lex} \vec x$.

If $x_{i+1} > a_{i+1}$, then $\phi(\vec x) > n$ because $a_{i+1}$ is the minimal such that $\phi(z_1,...,z_i-1,0,...,0) + a_{i+1}g_{i+1} \ge n$, as it was set using $\mathsf{decrementAndSolve}((z_1,...,z_i,0,...,0), i)$.

If $x_{i+1} = a_{i+1}$, then $x_j$ must be nonzero for some $i+2\le j \le d$, making $\phi(\vec x) > \phi(\vec a) \ge n$.

\end{proof}

We note that $\mathsf{nextCandidate}$ itself outputs a $previousCandidate$ which satisfies the assumption of Theorem 3.2, as it sets $a_d$ to either zero in step 4 or, if $d=i+1$, precisely the \quotes{minimal} value in step 11.


\begin{thm}
$\mathsf{nextCandidate}$ will reach the lexicographically next factorization in a finite number of steps, and thus will enumerate the entire factorization set in a finite number of steps.
\end{thm}
\begin{proof}
The output will always be less than the input, by virtue of decrementing a certain index and leaving all indices to the left unchanged; and the output space is a bounded subset of $\NN^d$, and thus finite, by virtue of all coordinate modifications in the algorithm either setting to zero, decrementing by one, or performing $\mathsf{decrementAndSolve}$, so each index $a_i$ will be set less than $n/g_i + 1$. By the pigeonhole principle it will reach any point in the output space, in particular the lexicographically next or final factorization, in a finite number of steps.
\end{proof}

\subsubsection{Optional modulo optimization for penultimate index}\hspace{-0.3em}\footnote{This is likely the \quotes{congruence-based algorithm},
mentioned in \cite{slopes} as a direct predecessor of the \quotes{slopes algorithm}, whose performance was benchmarked in \cite{performance}. }
Let the rightmost nonzero index $i$ from step 1 be $d-1$, such that the rest of the steps consist of decrementing index $d-1$ and solving for index $d$. If the last candidate was valid, meaning $\phi(...,a_{d-1},b) = 0$, then $(...,a_{d-1}-1,x)$ will have $\phi(...,a_{d-1}-1,x) \equiv -g_{d-1} \bmod g_d$ for all $x$. This means there will be no valid candidates until index $d-1$ decrements to $m$ where $m$ is the order of $a$ in $\ZZ_{g_d}$. Since invalid candidates do not affect the final result set of factorizations, this means an implementer may optionally choose to use the above logic in Algorithm 3.1 for the case $i=d-1$, skipping straight to $(...,a_{d-1}-m,x)$.

Similarly, if the last 2 generators have a common denominator $D$, and $\phi(...,a_{d-2},x,y)=0$, then $\phi(...,a_{d-2}-1,z,w) \equiv -g_{d-2} \bmod D$ for all $z,w$. The same applies if the last $n$ generators share a common denominator.

This was implemented in our reference CUDA implementation for index $d-1$ and runtimes were collected with it enabled vs disabled. We found significant runtime reductions in all cases when enabled.
\pagebreak
\subsection{Wrapper stream for returning only factorizations}

If one wanted to define a stream that produced actual factorizations, rather than factorization candidates, one would wrap this algorithm in another which only yields values when a valid factorization was produced, ignoring invalid candidates.

\begin{alg}[Wrapper stream for returning only factorizations]\ \\
$\mathsf{nextFactorization}: GreedyState \to GreedyState$\\
Input: $(prevState: GreedyState)$\\
Output: $GreedyState$

\begin{enumerate}
    \item $let\ result = \mathsf{nextCandidate}(prevState);$
    \item $while\ (!result.wasValid)\ result=\mathsf{nextCandidate}(result)); $
    \item Return $result$
\end{enumerate}
\end{alg}

Considered as GPU kernels directly (see GPU implentation section below), $\mathsf{nextCandidate}$ exhibits better SIMD (same instruction multiple data) performance compared to the $\mathsf{nextFactorization}$ function. This is because $\mathsf{nextCandidate}$ allows each iteration to return after only having decremented one coordinate and solved for one other, regardless of whether that successfully resulted in a factorization; as opposed to $\mathsf{nextFactorization}$ which performs that action in an unbounded while loop (meaning that some threads would be idle if their while loop finishes before others). For this reason we will work directly with the next candidate stream rather than the next factorization stream, such that checking validity of and saving the resulting candidates will be the responsibility of the calling code.

\section{Parallelization}

As lexicographic order is a total order on $\NN^d$, the overall result set can be partitioned into nonoverlapping subsets for each worker via the use of a \textbf{bound}, as mentioned in the above section on the stream data struct. Specifically, if an existing worker's previous candidate was $\vec a$ and existing bound was $\vec b$, the introduction of a new bound $\vec c$ lying in the range $(\vec a, \vec b)$ allows the existing worker's work slice to be bifurcated into the ranges $(\vec a, \vec c)$ and $(\vec c, \vec b)$ which will be taken by the existing and new workers respectively. If $\vec c$ is a valid factorization, it must be saved either now during bound splitting, or later as part of one of the workers' results.

The nonoverlapping nature of the resulting repeatedly bifurcated result sets allows the results to be consumed directly without any intermediate step to check a new result against existing results for duplication or minimality.

\subsection{Division of labor}

While there surely are many valid ways to find a new bound lying in an existing range, one very simple work division example would be, with last candidate $\vec a$, bound $\vec b$, new bound $\vec c$:

\begin{alg}[Example work-splitting algorithm]\ \\
$\mathsf{splitWork}: ParallelLexicographicState \to \NN^d$
\begin{enumerate}
\item Copy $\vec b$ to $\vec c$.
\item Set $1\leq i < d$ minimal such that $a_i > 0$ and $a_i \neq b_i$.
\item Set $c_j=0$ for $j \ge i+2$
\item Perform $\mathsf{decrementAndSolve}(\vec c, i)$
\end{enumerate}
\end{alg}

This is, in short, performing $\mathsf{decrementAndSolve}$ on the \emph{leftmost} suitable coordinate, instead of the \emph{rightmost} as in $\mathsf{nextCandidate}$.

This leaves the remainder of the solutions $\vec d$ with $d_j = a_j\ \forall 1 \leq j \leq i$ to the first worker, and the lets the second worker solve solutions lexicographically less than $\vec c$.

\begin{remark}
We note that the resulting $\vec c$ satisfies the assumption of Theorem 3.2 as it sets $c_d$ to either zero or, if $i+1=d$, it sets $c_{i+1}$ to the required \quotes{minimal} value. This guarantees that $\mathsf{nextCandidate}$ continues to return either the next factorization or a candidate greater than the next factorization, when given a $previousCandidate$ which was the result of Algorithm 4.1 instead of Algorithm 3.1.
\end{remark}

We do not prove that the above work-splitting function will always succeed in producing a new bound lying  strictly between the existing worker's previous candidate and bound. After all, if the number of workers exceeds the number of remaining $\mathsf{nextCandidate}$ stream values, the pigeonhole principle dictates that it is impossible to assign each worker a nonempty result set. We caution implementers of future work division functions that it is important to check whether the proposed new bound does indeed lie in the range $(\vec a, \vec b)$, discarding the proposed new bound if not.

Various additional heuristics could be used to optimize desired properties of the new bound, which may take into account relative worker processing power, the approximate expected result size of each resulting slice, the number of workers wishing to receive new work at once, etc.\footnote{See \cite[Section 2.3 Algorithmic Design Issues]{branchandbound} for more discussion of the \quotes{initial generation and allocation of work units} and \quotes{subsequent allocation and sharing of work units}, where \quotes{the objectives of such a policy should be to balance the workload among processes}.} Our implementation is agnostic to these concerns: it keeps the \quotes{original} work slice at the tail of its list of work slices; work splitting starts at the tail and works backwards; and each work slice $i$, if it needs fresh work, tried to split slice $i+1, i+2,...$ until it succeeds or reaches the tail.

\subsection{Distributed and multi-core CPU systems}
In addition to the reference GPU implementation we describe below, the above technique of bounded lexicographic enumeration is directly applicable to distributed environments and multithreaded CPU systems. Rather than \quotes{workers} being GPU cores operating on distinct slices of a GPU memory grid, they are distributed nodes operating on the node's host memory or CPU cores operating on the CPU's memory, respectively.

\section{GPU implementation}
As GPUs offer many cores on which work can simultaneously operate, GPUs are well suited for handling parallel tasks. We provide a reference CUDA implementation in the supplementary code files.

A GPU \textbf{kernel} is a function defined to run on one \emph{thread} in a \emph{block} of a \emph{grid} on a GPU. When the function is called from the CPU with certain grid parameters, it launches the kernel on many threads in each block on the grid in parallel, splitting blocks among the GPU's SMs (streaming multiprocessors), blocking the CPU call until all threads in the grid have completed the kernel.

In our implementation, a grid of $ParallelLexicographicState$ structs is allocated on the device. Initially, $\mathsf{splitWork}$ is called for each grid element to attempt to split work from some existing state. Then $\mathsf{nextCandidate}$ is repeatedly run as a CUDA kernel, operating on all grid elements simultaneously. Then the results will be checked for validity and saved if valid (although, see the below section on device buffers). Work splitting reoccurs periodically (see section 5.1.2).

The matter of choosing optimal the grid characteristics (threads per block and blocks per thread) greatly affects overall runtime, but the optimal values are highly dependent on both the particular hardware used and the particular problem chosen (dimension and generators).

\subsection{Transfer bottleneck}
GPU computing remains bottlenecked by relatively slow transfer speeds between host and device memory (CPU and GPU respectively). It is therefore worthwhile to attempt to minimize the frequency with which this bottleneck is encountered.

\subsubsection{Large device buffer for each worker}
If the ratio of device and host memory to grid size involved is comfortably high, it is feasible to implement a method of utilizing separate on-device buffers for each worker, which will be dumped when any worker's buffer becomes full. As the buffer is only written one element at a time (per thread), it is advantageous to make the device buffers large (perhaps 1000 factorizations per grid element) to allow as many consecutive kernel calls as possible before needing to copy from device buffer to host.

\subsubsection{Running multiple kernels between bound splitting}
If work splitting happens on the host (as it does in the provided reference implementation), it involves copying all states to and from the device, which is a relatively expensive task if attempted to be performed after every single $\mathsf{nextCandidate}$ kernel call. A number $kernelCallsBetweenFreshBounds$ can be used to call several (perhaps 1024) kernels in a row before attempting to give all the workers fresh bounds, allowing a higher amount of work to be performed before requiring any host-device transfers.

\subsection{Performance}
Table 1 shows the runtimes of both our CUDA implementation and the GAP \cite{GAP} function $\mathsf{FactorizationsElementWRTNumericalSemigroup}$ \cite{numericalsgps} on a fixed set of inputs. The CUDA implementation was run with the \quotes{modulo} optimization disabled and enabled respectively for comparison on select inputs.

The runtimes show that compared to the GAP implementation, our CUDA implementation is comfortably faster in low dimensions, specifically $d \le 5$, with the advantage increasing as the element $n$ increases. The GAP function was faster in dimension 5 for small elements, but became slower as the elements became larger. In dimensions greater than 5, the GAP function was faster in all measured cases.

\begin{table}
{\footnotesize
\begin{tabular}{|c|c|c|c|c|c|c|}
\hline
Generators & $n$ & $|\mathsf Z(n)|$ & CUDA (s) & CUDA w/ modulo (s) & GAP 4.13 \cite{numericalsgps} (s)  \\
\hline
13,37,38 & 1000&30&0&&0\\
&20000&10991&0&&0.9\\
&45000&55503&1&&9\\
&70000&134209&1&&36\\
&150000&615856&4&&620\\
&225000&1385404&10&&\\
&300000&2462699&18&&\\
&500000&6840027&52&&\\
\hline
...,40&1000&274&1&1&0\\
&5000&29601&2&1&0.75\\
&9000&169752&4&2&6\\
&13000&508263&9&5&34\\
&17000&1132667&20&9&112\\
&20000&1841247&31&16&163\\
&23000&2796813&44&22&279\\
&27000&4518931&66&34&635\\
&45000&20861676&197&132&OOM\\
\hline
...,41&1000&1920&1&1&0.015\\
&3000&125780&13&7&2.4\\
&5000&928872&91&29&28\\
&7000&3501274&240&69&143\\
&9000&9466814&&131&486\\
\hline
...,42&1000&10873&&4&0.1\\
&1500&70427&&19&2\\
&2000&273456&&53&4\\
&3000&1910466&&335&40\\
\hline
...,43&1000&52036&&27&0.5\\
&1500&473670&&201&6.7\\
&2000&2369185&&760&38\\
\hline
\end{tabular}
}
\medskip
\caption{Runtimes in seconds for computing $\mathsf Z(n)$. Computations were performed on a machine with an AMD Ryzen 3900X CPU and an NVIDIA RTX 3080 GPU. CUDA implementation parameters were $68*2$ grid size (the 3080 has 68 SMs), 8 threads per block, 1000 buffer per thread, and 1024 $kernelsBetweenNewBounds$. Our implementation disables the modulo optimization in dimension 3.\linebreak \linebreak We note that while the GAP function can run out of memory for extremely large result sets, our implementation does not have a comparable memory limitation as it does not need to store the result set in memory, as each iteration of $\mathsf{nextCandidate}$ only needs access to the previous iteration's returned state struct.}
\end{table}

\section{Future work}

\subsection{Dynamic behavior}

In $\mathsf{nextCandidate}$, if e.g. $i=d-3$, after decrementing $a_i$ and zeroing $a_d$, all of the remaining factorizations $\vec c$ of the form $(a_1,...,a_i,c_{d-2},c_{d-1},c_d)$ are in bijection with $\Zz(n-\phi(\vec a), (g_{d-1},g_{d-1},g_d))$. If $\Zz(x, (g_{d-2},g_{d-1},g_{d}))$ is stored for all $0 \le x \le m$ for some small value of $m$, and these sets are made accessible to each worker during their $\mathsf{nextCandidate}$ evaluation, then when index $i=d-3$ and $\phi(\vec a) \le m$, the remaining solutions $\vec c$ with $c_j = a_j$ for all $1 \le j \le d-3$ could be returned all at once by copying $\Zz(n-\phi(\vec a), (g_{d-2},g_{d-1},g_{d}))$ and prepending $(a_1,...,a_i)$ to each. (The signature and behavior of the $\mathsf{nextCandidate}$ function and its calling code would need to be modified appropriately to allow an output of more than one result.)

We give the above example of dynamic behavior for $i=d-3$, but it is valid for any number of the trailing indices.

\subsection{Work division improvement}
As noted in the Work division section, it may be possible to improve the division of labor between workers, perhaps increasing the average number of $nextCandidate$ stream results returned per worker before reaching its assigned bound, which would improve overall throughput as work reassignment only happens every $kernelsBetweenFreshBounds$ kernels (1024 in our implementation).

\subsection{Hybrid CPU-GPU work pool}
The provided CUDA implementation uses only one CPU thread, which is blocked for the duration of the next candidate kernel calls (and does not perform any $\mathsf{nextCandidate}$ calls itself, only work-splitting). As modern CPUs can run many threads in parallel (the Ryzen 3900X which evaluated the performance comparisons has 12 cores offering 24 simultaneous threads, 23 of which of which were sitting idle), it would be possible to simultaneously run workers on many CPU threads as well, receiving work slices during the work division steps and running the next candidate function while the GPU kernels are running.




\end{document}